\documentclass[12pt,a4paper]{amsart}
\usepackage[latin1]{inputenc}
\usepackage{eucal}
\usepackage{amsmath}
\usepackage{amsthm}

\DeclareMathOperator{\rank}{rank}

\DeclareMathOperator{\HF}{\mathcal{H}}
\DeclareMathOperator{\AHF}{H}
\DeclareMathOperator{\LM}{LM}
\DeclareMathOperator{\LT}{LT}
\DeclareMathOperator{\im}{im}

\title{A generalization of the Bombieri-Pila determinant method}
\author{Oscar Marmon}
\date{}

\begin{document}

\newcommand{\epsi}{\varepsilon}
\newcommand{\xx}{\mathbf{x}}
\newcommand{\xxi}{\boldsymbol{\xi}}
\newcommand{\eeta}{\boldsymbol{\eta}}
\newcommand{\0}{\boldsymbol{0}}
\newcommand{\bb}{\mathbf{b}}
\newcommand{\uu}{\mathbf{u}}
\newcommand{\ee}{\mathbf{e}}
\newcommand{\vv}{\mathbf{v}}
\newcommand{\yy}{\mathbf{y}}
\newcommand{\zz}{\mathbf{z}}
\newcommand{\ww}{\mathbf{w}}
\newcommand{\cc}{\mathbf{c}}
\newcommand{\hh}{\mathbf{h}}
\renewcommand{\gg}{\mathbf{g}}
\newcommand{\bK}{\mathbf{K}}
\newcommand{\bB}{\mathbf{B}}
\newcommand{\ZZ}{\mathbb{Z}}
\newcommand{\Zpol}{\ZZ[X_1,\ldots,X_n]}
\newcommand{\FF}{\mathbb{F}}
\newcommand{\RR}{\mathbb{R}}
\newcommand{\NN}{\mathbb{N}}
\newcommand{\CC}{\mathbb{C}}
\renewcommand{\AA}{\mathbb{A}}
\newcommand{\PP}{\mathbb{P}}
\newcommand{\GG}{\mathbb{G}}
\newcommand{\QQ}{\mathbb{Q}}
\newcommand{\sB}{\mathsf{B}}
\newcommand{\cP}{\mathcal{P}}
\newcommand{\cR}{\mathcal{R}}
\newcommand{\cB}{\mathcal{B}}
\newcommand{\cC}{\mathcal{C}}
\newcommand{\cN}{\mathcal{N}}
\newcommand{\cM}{\mathcal{M}}
\newcommand{\cD}{\mathcal{D}}
\newcommand{\cA}{\mathcal{A}}
\newcommand{\cK}{\mathcal{K}}
\newcommand{\cF}{\mathcal{F}}
\newcommand{\cZ}{\mathcal{Z}}
\newcommand{\cX}{\mathcal{X}}
\newcommand{\cG}{\mathcal{G}}
\newcommand{\ud}{\mathrm{ud}}
\newcommand{\Zar}{\mathrm{Zar}}

\newtheorem{lemma}{Lemma}[section]
\newtheorem*{lemma*}{Lemma}
\newtheorem{prop}{Proposition}[section]
\newtheorem{thm}{Theorem}[section]
\newtheorem{claim}{Claim}[section]
\newtheorem{cor}{Corollary}[section]
\theoremstyle{remark}
\newtheorem*{notation*}{Notation}
\newtheorem*{note*}{Note}
\newtheorem{note}{Note}
\newtheorem*{rem*}{Remark}
\newtheorem{rem}{Remark}[section]
\newtheorem*{acknowledgement*}{Acknowledgement}
\theoremstyle{definition}
\newtheorem*{def*}{Definition}

\begin{abstract}
The so-called determinant method was developed by Bombieri and Pila in 1989 for counting integral points of bounded height on affine plane curves. In this paper we give a generalization of that method to varieties of higher dimension, yielding a proof of Heath-Brown's ``Theorem 14'' by real-analytic considerations alone.   
\end{abstract}

\maketitle


\section{Introduction}

In their paper \cite{Bombieri-Pila}, Bombieri and Pila used a determinant argument to get uniform estimates for the density of integral points on affine plane curves. Very roughly, the basic idea of the method can be described as follows: one constructs a collection of auxiliary polynomials, each vanishing at every integral point on the intersection of the curve with a small square. This is done by forming a determinant of a suitable set of monomials evaluated at the integral points, which vanishes if the square is chosen small enough.

Heath-Brown \cite{Heath-Brown02} developed a non-Archimedean version of the determinant method for counting rational points on projective hypersurfaces in any dimension. Whereas the original method of Bombieri and Pila can be said to use the Archimedean absolute value to estimate the determinant, Heath-Brown instead groups together points that are $p$-adically close (for suitable primes $p$) and uses the $p$-adic absolute value to estimate the determinant.
Salberger \cite{Salberger07} has refined this method further, and it has turned out to be a very valuable tool for counting rational and integral points on algebraic varieties. One important feature of the method is that it produces uniform estimates for varieties of a given degree. 

In this paper, we will revisit the determinant method in the real setting, and try to generalize it to higher dimensional varieties. This has been done to some extent, for example by Pila \cite{Pila04} and Pila and Wilkie \cite{Pila-Wilkie}, but the focus there is on non-algebraic sets, whereas we consider algebraic varieties. We shall recover the main theorem in \cite{Heath-Brown02} using only real-analytic considerations. 


Thus, let $X \subset \AA^n_\RR$ be an irreducible closed subvariety. The main object of interest is, for a given positive real number $B$, the set of integral points of height at most $B$ on $X$,
\[
X(\ZZ,B) = \{\xx = (x_1,\dotsc,x_n) \in X\cap \ZZ^n ; |\xx|\leq B\},
\]
where $|\xx| = \max \{|x_1|,\dotsc,|x_n|\}$. Its cardinality is denoted by $N(X,B)$. 
The following is our main result for affine varieties.

\begin{thm}
\label{thm:affine}
Let $X \subset \AA^n_\RR$ be an irreducible closed subvariety of dimension $m$ and degree $d$. Let $I = I(X) \subset \RR[x_1,\dotsc,x_n]$ be the ideal of $X$. Then, for any $\epsi >0$, there exists a polynomial $g \in \ZZ[x_1,\dotsc,x_n] \setminus I$ of degree 
\[
k \ll_{n,d,\epsi} B^{m d^{-1/m}+\epsi},
\]
all of whose irreducible factors have degree $O_{n,d,\epsi}(1)$, such that $g(\xx) = 0$ for each $\xx \in X(\ZZ,B)$.
\end{thm}

\begin{rem*}
By the dimension of $X$, we mean its dimension as an irreducible algebraic subset of $\RR^n$ (see \cite{Bochnak-Coste-Roy}). Thus it equals the Krull dimension of the ring $\RR[x_1,\dotsc,x_n]/I(X)$. It need not, however, equal the dimension of the variety $X_\CC \subset \AA^n_\CC$ defined by the same equations, nor the dimension of the ring $\RR[x_1,\dotsc,x_n]/J$ for an arbitrary ideal $J$ defining $X$.

We define the degree of $X$ as the degree of its projective closure (see Section \ref{sec:hilbert}).
\end{rem*}

Theorem \ref{thm:affine} allows us to use whatever estimates that may be available for varieties of dimension $m-1$ to bound $N(X,B)$. In case $m=1$, the theorem produces a collection of curves of bounded degree harbouring all points of $X(\ZZ,B)$. Using Bézout's theorem, one concludes that $N(X,B) = O_{n,d,\epsi}(B^{1/d+\epsi})$, as in \cite{Bombieri-Pila}.

Theorem \ref{thm:affine} is a corollary to our main result for projective varieties. The following notation will be used. We shall sometimes omit the ground field $\RR$ and write $\PP^n$ for $\PP^n_\RR$. The homothety class of $\xx \in \RR^{n+1}$ is denoted by $[\xx]$. The set of rational points $\PP^n(\QQ)$ consists of the points $x \in \PP^n$ for which we can find a representative $\xx \in \QQ^{n+1}$ such that $[\xx] = x$. If $X \subset \PP^n$ is a locally closed subset, we define $X(\QQ) = X \cap \PP^n(\QQ)$.
 
We define the \emph{height} of a rational point $x \in \PP^n(\QQ)$ as follows. Choose a representative $\xx \in \ZZ^{n+1}$ for $x$ such that $\gcd(x_0,\dotsc,x_n)=1$ and put $H(x) = \max\{|x_0|,\dotsc,|x_n|\}$. The density of rational points on $X$ is measured by the counting function $N(X,\cdot)$ given by
\[
N(X,B) = \#\{x \in X(\QQ); H(x) \leq B\}
\]
for $B \in \RR_{\geq 0}$. More generally, we wish to impose individual restrictions on the coordinates $x_0,\dotsc,x_n$. Let $\bB = (B_0,\dotsc,B_n)$ be an $(n+1)$-tuple of positive real numbers. Then we define
\[
S(X,\bB) = \{\xx \in \ZZ^{n+1}; [\xx] \in X(\QQ), |x_i| \leq B_i, i=0,\dotsc,n\}.
\]

Following Broberg \cite{Broberg02}, we use graded monomial orderings to state our main result. These will be defined in Section \ref{sec:hilbert}. Given a monomial ordering $<$ and an irreducible projective variety $X \subseteq \PP^n$, we associate to the pair $(<,X)$ an $(n+1)$-tuple $(a_0,\dotsc,a_n)$ of real numbers satisfying $0 \leq a_i \leq 1$ and $a_0 + \dotsb + a_n = 1$ (see Proposition \ref{prop:a_i}).   

\begin{thm}
\label{thm:auxiliary}
Let $X \subset \PP^n_\RR$ be an irreducible closed subvariety of dimension $m$ and degree $d$, with ideal $I = I(X) \subset \RR[x_0,\dotsc,x_n]$. Let $\bB = (B_0,\dotsc,B_n)$ be an $(n+1)$-tuple of positive real numbers, and let $<$ be a graded monomial ordering on $\RR[x_0,\dotsc,x_n]$. Then, for any $\epsi >0$, there is a homogeneous polynomial $G \in \ZZ[x_0,\dotsc,x_n] \setminus I$ of degree
\[
k \ll_{n,d,\epsi} (B_0^{a_0} \dotsb B_n^{a_n})^{(m+1)d^{-1/m}+\epsi}, 
\]
all of whose irreducible factors have degree $O_{n,d,\epsi}(1)$, such that $G(\xx) = 0$ for all $\xx \in S(X,\bB)$.  
\end{thm}

We stress that this is not a new result (cf. Heath-Brown's original theorem \cite[Thm 14]{Heath-Brown02} and Broberg's generalization \cite[Thm 1]{Broberg02}). Theorem \ref{thm:auxiliary} is proven in Section \ref{sec:projective}.

\begin{acknowledgement*}
I thank my supervisor Per Salberger for suggesting this topic of research. I am grateful to Jonathan Pila for his suggestions that greatly improved this paper, and to Tim Browning who has shown interest in my work.
\end{acknowledgement*}
 

\section{Estimating a determinant}
\label{sec:determinant}

In this section, following Pila \cite{Pila04}, we shall derive an estimate that will play a key role in the proof of Theorem \ref{thm:auxiliary}. We shall follow the notation used in \cite[\S 3]{Pila04}. For a multi-index $\alpha = (\alpha_1,\dotsc,\alpha_m) \in \ZZ_{\geq 0}^m$ we write
\begin{gather*}
|\alpha| = \alpha_1 + \dotsb + \alpha_m, \quad
\alpha! = \alpha_1! \dotsb \alpha_m!.
\end{gather*}
For $\zz, \yy \in \RR^m$, write $\zz - \yy = (z_1 - y_1,\dotsc,z_m - y_m)$. Furthermore, if $\zz \in \RR^m$ and $\alpha \in \ZZ_{\geq 0}^m$, we define $\zz^\alpha = z_1^{\alpha_1} \dotsb z_m^{\alpha_m}$.
Define the sets
\begin{gather*}
\Lambda_m(k) = \{ \alpha \in \ZZ_{\geq 0}^{m}; |\alpha| = k \},\\
\Delta_m(k) = \{ \alpha \in \ZZ_{\geq 0}^{m}; |\alpha| \leq k \} = \bigcup_{i=0}^k \Lambda_m(i)
\end{gather*}
and put $L_m(k) = \#\Lambda_m(k)$, $D_m(k) = \#\Delta_m(k)$. Then we have
\begin{gather*}
L_m(k) = \binom{k+m-1}{m-1}, \\
D_m(k) = \sum_{i=0}^{k} L_m(i) = \binom{k+m}{m}.
\end{gather*}

Let $U\subset \RR^m$ be a compact subset. Consider a function $\psi:U\to\RR$. For a multi-index $\alpha$ we write
\[
\partial^\alpha \psi = \frac{\partial^{\alpha_1}}{\partial x_1^{\alpha_1}} \dotsm \frac{\partial^{\alpha_n}}{\partial x_n^{\alpha_n}} \psi.
\] 
If $\partial^\alpha \psi$ is defined and continuous for all $\alpha$ with $|\alpha| \leq k$ (in which case we write $\psi \in C^k(U)$), define the quantity
\[
 \Vert \psi \Vert_k = \max_{\xx \in U} \max_{0\leq |\alpha|\leq k} \left|\partial^\alpha \psi(\xx)\right|.
\]
We shall investigate how $\Vert \cdot \Vert_k$ behaves under multiplication of functions.

\begin{prop}
\label{prop:monomialnorm}
Let $U\subset \RR^m$ be a compact subset and suppose that $\phi_1,\dotsc,\phi_\ell \in C^k(U)$. Then we have 
\begin{equation}
\label{eq:monomialnorm}
\Vert \phi_1 \dotsb \phi_\ell \Vert_k \leq \ell^k
\Vert \phi_1 \Vert_k \dotsb \Vert\phi_\ell\Vert_k.
\end{equation}
\end{prop}

\begin{proof}
Suppose that $\alpha$ is a multi-index with $|\alpha| \leq k$. Using the product rule, we expand $\partial^\alpha (\phi_1 \dotsb \phi_\ell)$ into a sum of $\ell^{|\alpha|}$ terms of the form
\[
(\partial^{\alpha^1} \phi_1) \dotsm (\partial^{\alpha^\ell} \phi_\ell),
\]
where $\alpha^j \in \ZZ_{\geq 0}^m$ satisfy $\alpha^1 + \dotsb +\alpha^\ell = \alpha$. Clearly we have
\[
\sup_{\xx \in U} |(\partial^{\alpha^1} \phi_1) \dotsm (\partial^{\alpha^\ell} \phi_\ell)| \leq 
\Vert \phi_1 \Vert_k \dotsb \Vert\phi_\ell\Vert_k,
\]
so the estimate \eqref{eq:monomialnorm} holds.
\end{proof}

The following result is a slight modification of Lemma 3.1 in \cite{Pila04}.

\begin{lemma}
\label{lem:GeneralDeterminant}
Let $\mu \in \NN$. Let $U \subset \RR^m$ be a compact convex set of diameter $r<1$, and let $\psi_1,\dotsc,\psi_\mu$ be real-valued $C^\nu$ functions on $U$. Given $\mu$ distinct points $\xx^{(1)}, \xx^{(2)}, \dotsc,\xx^{(\mu)}$ in $U$, define the $\mu \times \mu$-determinant
\[
 \Delta = \det (\psi_i(\xx^{(j)})).
\]
Suppose that $\nu \in \NN$ satisfies $D_{m}(\nu-1) \leq \mu \leq D_{m}(\nu)$, and define
\[
 e = \sum_{i=0}^{\nu-1} i L_{m}(i) + \nu (\mu - D_{m}(\nu -1)).
\]
Then we have
\begin{equation}
\label{eq:GeneralDeterminant}
 |\Delta| \leq  \mu! D_{m}(\nu)^\mu \prod_{i=1}^\mu \Vert \psi_i \Vert_\nu \cdot r^e.
\end{equation}
\end{lemma}
\begin{proof}
We write each entry of $\Delta$ in the following form, using a Taylor expansion of order $\nu - 1$ around $\xx^{(1)}$:
\begin{equation}
\label{eq:Taylor}
\psi_i(\xx^{(j)}) = \sum_{\alpha \in \Delta_m(\nu -1)} \frac{\partial^\alpha \psi_i(\xx^{(1)})}{\alpha!} (\xx^{(j)}-\xx^{(1)})^\alpha + R_\nu(\xx^{(j)}). 
\end{equation}
Then there exists a point $\xxi_{ij}$ on the line segment joining $\xx^{(1)}$ and $\xx^{(j)}$ such that
\[
 R_\nu(\xx^{(j)}) = \sum_{\alpha \in \Lambda_m(\nu)} \frac{\partial^\alpha \psi_i(\xxi_{ij})}{\alpha!} (\xx^{(j)}-\xx^{(1)})^\alpha.
\]
The number of terms in each such sum is $D_{m}(\nu)$. Write each column in $\Delta$ as a sum of $D_{m}(\nu)$ column vectors according to \eqref{eq:Taylor}, where each vector corresponds to a specific $\alpha \in \ZZ_{\geq 0}^{m}$. This gives an expansion of $\Delta$ as a sum of $D_{m}(\nu)^\mu$ determinants $\Delta_\ell$ of the same dimension as $\Delta$. However, any of these determinants possessing more than $L_m(i)$ columns corresponding to monomials of order $i$ in $\xx^{(j)}-\xx^{(1)}$, for any $i\leq \nu-1$, has to vanish, since then these columns are linearly dependent. Consequently, the determinants $\Delta_\ell$ satisfy
\[
|\Delta_\ell| \leq \mu! \prod_{i=1}^\mu \Vert \psi_i \Vert_\nu \cdot r^e.
\]
Summing over $1\leq \ell \leq D_{m}(\nu)^\mu$ we get \eqref{eq:GeneralDeterminant}.
\end{proof}


\section{Monomial orderings and Hilbert functions}
\label{sec:hilbert}

We shall review the basics of Hilbert functions, following mainly the exposition in \cite{CLO}. Let $K$ be a field. If $\alpha = (\alpha_0,\dotsc,\alpha_n) \in \ZZ_{\geq 0}^{n+1}$, then we shall write $\xx^\alpha$ for the monomial $x_0^{\alpha_0} \dotsb x_n^{\alpha_n}$.

\begin{def*}
A \emph{graded monomial ordering} on $K[x_0,\dotsc,x_n]$ is a total ordering $<$ on $\ZZ^{n+1}_{\geq 0}$ (or, equivalently, on the set of monomials $\{\xx^\alpha; \alpha \in \ZZ^{n+1}_{\geq 0}\}$) satisfying
\begin{enumerate}
\item
$\alpha \geq \0$ for any $\alpha \in \ZZ^{n+1}_{\geq 0}$;
\item
if $\alpha, \beta, \gamma \in \ZZ^{n+1}_{\geq 0}$ and  $\alpha < \beta$, then $\alpha + \gamma < \beta + \gamma$;
\item
if $\alpha, \beta \in \ZZ^{n+1}_{\geq 0}$ and $\alpha \leq \beta$, then $|\alpha| \leq |\beta|$.
\end{enumerate}
\end{def*}

Given a graded monomial ordering $<$, we can define the \emph{leading monomial} and \emph{leading term} of a polynomial
\[ 
f = \sum c_\alpha \xx^\alpha \in K[x_0,\dotsc,x_n]
\] 
as
\[
\LM(f) = \xx^\beta, \quad \LT(f) = c_\beta \xx^\beta, \text{ where } \beta = \max \{\alpha \in \ZZ^{n+1}_{\geq 0}; c_\alpha \neq 0\}.
\]

For $s \in \ZZ_{\geq 0}$, let $K[x_0,\dotsc,x_n]_s$ be the $K$-vector space of homogeneous polynomials of degree $s$ (including the zero polynomial). A basis for $K[x_0,\dotsc,x_n]_s$ is given by the monomials $\xx^\alpha$, $\alpha \in \Lambda_{n+1}(s)$.


If $I \subseteq K[x_0,\dotsc,x_n]$ is a homogeneous ideal, let 
\[
I_s = I \cap K[x_0,\dotsc,x_n]_s,
\]
a $K$-subspace of $K[x_0,\dotsc,x_n]_s$. Define the \emph{Hilbert function} of $I$, 
\[
\HF_I: \ZZ_{\geq 0} \to \ZZ_{\geq 0},
\] 
by
\[
\HF_I(s) = \dim_K \left(K[x_0,\dotsc,x_n]_{s}/I_{s}\right).
\]
For any ideal $I\subseteq K[x_0,\dotsc,x_n]$, let $\LT(I)$ be the ideal generated by the leading terms of all the polynomials in $I$.


\begin{prop}
\label{prop:hilbertfunction}
Let $I \subset K[x_0,\dotsc,x_n]$ be a homogeneous ideal. Then $\HF_I = \HF_{\LT(I)}$. In other words, $\HF_I(s)$ is the number of monomials of degree $s$ that are not leading monomials of any $F \in I$.
\end{prop} 
\begin{proof}
\cite[Prop. 9, Ch. 9.3]{CLO}
\end{proof}

We shall also have use of the following functions, related to the Hilbert function. For each $i \in \{0,\dotsc,n\}$, define
\[
\sigma_{I,i}(s) = \sum_{\substack{\alpha \in \Lambda_{n+1}(s)\\\xx^\alpha \notin \LT(I)}} \alpha_i.
\]
From Proposition \ref{prop:hilbertfunction} we see that
\begin{equation}
\label{eq:sum-sigma}
\sigma_{I,0}(s) + \dotsb + \sigma_{I,n}(s) = s \HF_I(s).
\end{equation}

%

Let $I \subset K[x_0,\dotsc,x_n]$ be a homogeneous ideal, and suppose that $I$ is generated by polynomials of degree at most $\delta$. Then there are polynomials $P_I, Q_{I,0},\dotsc,Q_{I,n} \in \QQ[t]$, and an integer $s_0$, depending only on $n$ and $\delta$, such that for every $s \geq s_0$ we have
\begin{gather*}
\HF_I(s) = P_I(s), \ \sigma_{I,0}(s) = Q_{I,0}(s), \dotsc, \sigma_{I,n}(s) = Q_{I,n}(s).
\end{gather*}
Furthermore, the coefficients of $P_I, Q_{I,0},\dotsc,Q_{I,n}$ are also bounded in terms of $n$ and $\delta$. A proof of these statements can be found in \cite{Broberg02}. $P_I$ is called the \emph{Hilbert polynomial} of $I$.

Let $X \subseteq \PP^n_K$ be an irreducible variety, and let $I = I(X)$ be its ideal. We define the Hilbert function and Hilbert polynomial of $X$ by $\HF_X = \HF_I$, $P_X = P_I$. The degree of $P_X$ equals the dimension of $X$. Furthermore, if $m = \dim X$ and 
\[
P_X(t) = \sum_{i=0}^m b_i t^i, 
\]
then the \emph{degree} of $X$ is $\deg X = m! b_m$. By \eqref{eq:sum-sigma} we see that each of the $Q_{I,i}$ has degree at most $m+1$.

\begin{rem}
\label{rem:Hilbert}
Given positive integers $d$ and $n$, there is a finite set $\cF_{n,d}$ of functions $\ZZ_{\geq 0} \to \ZZ$ such that $\HF_X \in \cF_{n,d}$ for all subvarieties $X \subset \PP^n_K$ of degree $d$. (See \cite[Lemma 1.4]{Salberger07} and the references listed there.) One can then prove that the ideal $I$ of $X$ is generated by polynomials of degree $O_{n,d}(1)$ \cite[Lemma 1.3]{Salberger07}. In particular, all the coefficients of $P_X$ and $Q_{I,0},\dotsc,Q_{I,n}$ can be bounded in terms of $\deg X$ and $n$.
\end{rem}

From equation \eqref{eq:sum-sigma} and the following discussion, we draw the following conclusion.
\begin{prop}
\label{prop:a_i}
Let $X \subseteq \PP^n_K$ be an irreducible variety of degree $d$, and let $I$ be its ideal. Then there are numbers $a_{I,i} \in [0,1]$ for $i \in \{0,\dotsc,n\}$ such that
\begin{gather}
\frac{\sigma_{I,i}(s)}{s \HF_I(s)} = a_{I,i} + O_{n,d}(1/s) \text{ as } s \to \infty,\\
a_{I,0} + \dotsb + a_{I,n} = 1. 
\end{gather} 
\end{prop}

\begin{rem}[Hilbert function and degree for affine varieties]
\label{rem:projectiveHF}
For $s \in \ZZ_{\geq 0}$, let $K[x_1,\dotsc,x_n]_{\leq s}$ be the vector space of polynomials of degree at most $d$. Let $I \subseteq K[x_1,\dotsc,x_n]$ be an ideal. Let $I_{\leq s} = I \cap K[x_1,\dotsc,x_n]_{\leq s}$, a $K$-subspace of $K[x_1,\dotsc,x_n]_{\leq s}$. Then we define the \emph{affine Hilbert function} of $I$ to be  
\[
\AHF_I(s) = \dim_K \left(K[x_1,\dotsc,x_n]_{\leq s}/I_{\leq s}\right),
\]
Consequently, if $X \subseteq \AA^n_K$ is an irreducible variety, and $I$ its ideal, we define $\AHF_X = \AHF_I$. However, if $I^h \subseteq K[x_0,\dotsc,x_n]$ denotes the homogenization of $I$, and $\overline X \subseteq \PP^n_K$ the projective closure of $X$ (see \cite[Ch. 8.4]{CLO}), then by \cite[Thm. 12, Ch. 9.3]{CLO} we have $\AHF_X=\AHF_{I^h} = \HF_{I}=\HF_{\overline X}$. In particular, we can define the \emph{degree} of $X$ as 
\[
\deg X = \deg \overline X.
\]
\end{rem}

\section{Yomdin and Gromov's algebraic lemma}

In most applications of the determinant method, one uses the implicit function theorem to parametrize the points of a variety $X$. When applying a determinant estimate like Lemma \ref{lem:GeneralDeterminant}, we have to bound the sizes of the partial derivatives of the implicit functions. In the original paper \cite{Bombieri-Pila}, this is done by excising the subset where the partial derivatives blow up, and covering it with smaller boxes, in which one carries out the method again, recursively. However, in a recent paper \cite{Pila-Wilkie}, Pila and Wilkie employ a more powerful method of parametrization, due to Yomdin and Gromov. 

We extend the notions defined in Section \ref{sec:determinant} to vector-valued functions: a function $\phi:U \to \RR^n$, where as before $U$ is a compact subset of $\RR^m$, belongs to the class $C^k$ if all its coordinate functions do. In this case we define
\[
 \Vert \phi \Vert_k = \max_{\xx \in U} \max_{0\leq |\alpha|\leq k} \left|\partial^\alpha \phi(\xx)\right|,
\]
where $|(x_1,\dotsc,x_n)| = \max_i |x_i|$. The result we shall now state was first proven in a weaker form by Yomdin \cite{Yomdin87A}, and the final version was obtained by Gromov \cite{Gromov}. The proof in \cite{Gromov} is rather brief, but recently Burguet \cite{Burguet} has given a more detailed proof. Another proof can be extracted from that of Pila and Wilkie \cite{Pila-Wilkie} in the more general setting of definable sets.

\begin{lemma}
\label{lem:gromov}
Let $V \subset \AA^n_\RR$ be an algebraic variety of dimension $m<n$ and degree $d$, and let $Y = V \cap [-1,1]^n$. For each $r \in \ZZ_+$, there exists an integer $N_0$, depending only on $n$, $r$ and $d$, and $C^r$-functions $\phi_i : [-1,1]^m \to Y$ for $i=1,2,\dotsc,N_0$, such that
\[
\bigcup_{i=1}^{N_0} \phi_i([-1,1]^m) = Y \text{ and } \Vert \phi_i \Vert_r \leq 1.
\] 
\end{lemma}

Our formulation differs from Gromov's on two points. Firstly, we use $[-1,1]^m$ and $[-1,1]^n$ instead of $[0,1]^m$ and $[0,1]^n$. It is completely obvious that these two cases give equivalent statements. Secondly, we claim that $N_0$ is bounded in terms of the degree of $Y$, whereas Gromov uses the sum of the degrees of a collection of polynomials defining $Y$. 
To see that our statement follows from Gromov's, we appeal to Remark 3.1.

\section{Rational points on projective varieties}
\label{sec:projective}

In this section we shall prove Theorem \ref{thm:auxiliary}. To begin with, we have
\[
S(X,\bB) = \bigcup_{i=0}^n S_i(X,\bB),
\]
where $S_i(X,\bB)$ is the set of $\xx \in S(X,\bB)$ such that $|x_j/B_j| \leq |x_i/B_i|$ for all $j \in \{0,\dotsc,n\}$. Thus, let $i \in \{0,\dotsc,n\}$. We shall find a form vanishing at every point of $S_i(X,\bB)$. By permuting the variables, we can assume that $i=0$.

Let $\tau: \RR^{n+1} \to \RR^{n+1}$ be given by
\[
\tau(x_0,\dotsc,x_n) = \left( \frac{1}{B_0} x_0, \dotsc, \frac{1}{B_n} x_n \right).
\]
$\tau$ induces an automorphism $\bar\tau: \PP^n_\RR \to \PP^n_\RR$. Let $Z \subset \PP^n_\RR$ be the image of $X$ under $\bar\tau$. Furthermore, let 
\[
T_0(X,\bB) = \tau (S_0(X,\bB)) \subseteq [-1,1]^{n+1}.
\]

Let $\iota_0 : \AA^n_\RR \to \PP^n_\RR$ be the open immersion given by
\[
(z_1,\dotsc,z_n) \mapsto [(1,z_1,\dotsc,z_n)],
\]
and let $Z_0 = \iota_0^{-1}(Z)$. By Remark \ref{rem:projectiveHF} we have $\deg Z_0 = \deg Z = \deg X$. Note also that $\left\vert \iota_0^{-1}([\yy]) \right\vert \leq 1$ if $\yy \in T_0(X,\bB)$.
 
For any $\delta \in \ZZ_{\geq 0}$, define
\[
M(\delta) = \{\ee \in \ZZ_{\geq 0}^{n+1}; |\ee| = \delta, \xx^\ee \notin \LT(I) \}.
\]
We say that a polynomial $G$ is \emph{defined in $M(\delta)$} if $\ee \in M(\delta)$ for all monomials $\xx^\ee$ occurring in $G$. If this is the case, then $G \notin I$.

Write $\sigma_i = \sigma_{I,i}(\delta)$, that is 
\begin{equation}
\label{eq:sigma_i}
\sigma_i = \sum_{\ee \in M(\delta)} e_i.
\end{equation}
 
Put $\mu = |M(\delta)|$, which by Proposition \ref{prop:hilbertfunction} equals $\HF_I(\delta)$, and choose the integer $\nu$ so that
\begin{equation}
\label{eq:choose_nu2}
D_{m}(\nu - 1) \leq \mu \leq D_{m}(\nu).
\end{equation}
Define
\begin{gather}
\label{eq:f}
f = \sum_{i=0}^{\nu-1} i L_m(i) + \nu (\mu - D_m(\nu-1)).
\end{gather}

By Yomdin and Gromov's lemma, there is a collection of $C^\nu$ functions 
\[
\phi_1,\dotsc,\phi_N:[-1,1]^m \to [-1,1]^n,
\]
such that $\Vert \phi_i \Vert_\nu \leq 1$, and whose images cover $Z_0 \cap [-1,1]^n$. The number of functions needed, $N$, depends only upon $n$, $d$ and $\delta$. Let $\phi \in \{\phi_1,\dotsc,\phi_N\}$, and define
\[
S_{0,\phi}(X,\bB) = \{\xx \in S_0(X,\bB); \bar\tau([\xx]) \in \iota_0(\im \phi)\}.
\]
Then we want to find a collection $\cG = \{G_1,\dotsc,G_k\}$ of homogeneous polynomials defined in $M(\delta)$ such that for each $\xx \in S_{0,\phi}(X,\bB)$, there is a polynomial $G \in \cG$ such that $G(\xx) = 0$. We shall prove that this is possible with $k$ satisfying
\begin{equation}
\label{eq:k1}
k \ll_{n,d,\delta} B_0^{m\sigma_0/f} \dotsb B_n^{m\sigma_n/f},
\end{equation}
where $\sigma_i$ and $f$ are as defined in \eqref{eq:sigma_i} and \eqref{eq:f}.
Multiplying these forms together, we get a form of degree $\delta k$, vanishing at each point of $S_{0,\phi}(X,\bB)$ but not belonging to $I$, and whose irreducible factors all have degree at most $\delta$.
Having proven this, we shall then choose $\delta$ as to give the desired estimate.

To prove the estimate \eqref{eq:k1}, we consider a covering of $[-1,1]^m$ by cubes $\Sigma$ of sidelength $\rho$, where the value of $\rho$ is to be chosen appropriately later on in the proof. For a certain fixed cube $\Sigma$, let $\xx^{(1)},\dotsc,\xx^{(q)}$ be an enumeration of the (finitely many) points in $S_{0,\phi}(X,\bB)$ such that $\bar\tau([\xx]) \in \iota_0(\phi(\Sigma))$. We shall then find a polynomial $G$, with integral coefficients and defined in $M(\delta)$, such that
 \begin{equation}
\label{eq:AuxPoly2}
G(\xx^{(1)}) = \dotsm = G(\xx^{(q)}) = 0.
\end{equation}
We shall achieve this by bounding the rank of the matrix
\begin{equation*}
\mathcal A = \left( (\xx^{(j)})^\ee \right)_{\substack{\ee \in M(\delta) \\ j=1,\dotsc,q}},
\end{equation*}
with rows corresponding to exponent $(n+1)$-tuples $\ee$, ordered according to the chosen graded monomial order $<$, and columns corresponding to the different points $\xx^{(j)}$. We will prove that $\mathcal A$ has rank at most $\mu-1$, thus producing a polynomial satisfying \eqref{eq:AuxPoly2}, as in \cite[Lemma 1]{Bombieri-Pila}. 

If $q \leq \mu - 1$, then of course $\rank \mathcal A \leq \mu - 1$. Assume therefore that $q \geq \mu$, and pick $\mu$ points among the $\xx^{(j)}$. Renumber the points as $\xx^{(1)},\dotsc,\xx^{(\mu)}$. We shall now examine the determinant
\[
\cD = \det \left( (\xx^{(j)})^\ee \right)_{\substack{\ee \in M(\delta) \\ j=1,\dotsc,\mu}}.
\]
Let $\yy^{(j)} = \tau(\xx^{(j)}) \in [-1,1]^{n+1}$. Then we have
\[
\cD = B_0^{\sigma_0} \dotsb B_n^{\sigma_n} \cD', \text{ where } \cD' = \det \left( (\yy^{(j)})^\ee \right)_{\substack{\ee \in M(\delta) \\ j=1,\dotsc,\mu}}.
\]
Writing $\zz^{(j)} = (y^{(j)}_1/y^{(j)}_0,\dotsc,y^{(j)}_n/y^{(j)}_0)$, we have $\zz^{(j)} \in \phi(\Sigma)$ by assumption. Furthermore, we have
\[
\cD' = \left(y^{(1)}_0 \dotsb y^{(\mu)}_0 \right)^\delta \Delta, \text{ where } \Delta = \det \left( (1,\zz^{(j)})^\ee \right)_{\substack{\ee \in M(\delta) \\ j=1,\dotsc,\mu}},
\]
so $|\cD'| \leq |\Delta|$. To estimate $\Delta$, we can now use Lemma \ref{lem:GeneralDeterminant}. Enumerating the elements $\ee(1),\dotsc,\ee(\mu)$ of $M(\delta)$ as specified by $<$, we let the functions $\psi_i:\Sigma \to \RR$ in the hypothesis of the lemma be given by
\[
\psi_i(\uu) = (1,\phi(\uu))^{\ee(i)}.
\]
Using Proposition \ref{prop:monomialnorm}, together with the fact that $\Vert \phi \Vert_\nu \leq 1$, we see that $\Vert \psi_i \Vert \ll_{n,d,\delta} 1$. Thus we get $|\Delta| \ll_{n,d,\delta} \rho^f$, so that 
\[
|\cD| \ll_{n,d,\delta} B_0^{\sigma_0} \dotsb B_n^{\sigma_n} \rho^f.
\]

Now we make use of the crucial fact that the entries in $\cD$ are all integers, from which follows that either $\cD=0$ or $|\cD| \geq 1$. We see that if we choose
\[
B_0^{-\sigma_0/f} \dotsb B_n^{-\sigma_n/f} \ll_{n,d,\delta} \rho \ll_{n,d,\delta} B_0^{-\sigma_0/f} \dotsb B_n^{-\sigma_n/f},
\]
we will get $|\cD|<1$, and consequently $\cD=0$. Thus, for each small cube $\Sigma$ we produce a polynomial $G$, defined in $M(\delta)$, with the desired property. To cover $[-1,1]^m$ we do not need more than 
\[
O_{n,d,\delta}(B_0^{m\sigma_0/f} \dotsb B_n^{m\sigma_n/f})
\]
of the cubes $\Sigma$, so the estimate \eqref{eq:k1} follows.

We have now proven that there exists a form $G \notin I$ of degree $k$ vanishing at every point in $S(X,\bB)$, where $k$ satisfies \eqref{eq:k1}. Furthermore, the irreducible factors of $G$ have degree at most $\delta$. It remains to choose the value of the parameter $\delta$. By Proposition \ref{prop:a_i}, we have
\[
\sigma_i = a_i \delta \HF(\delta) + O_{n,d}(\delta^m) = \frac{a_i d}{m!} \delta^{m+1} + O_{n,d}(\delta^m).
\]
as $\delta \to \infty$. Furthermore, our choice \eqref{eq:choose_nu2} implies that 
\[
\nu = d^{1/m} \delta + O_{n,d}(1),
\]
so we have
\begin{align*}  
f &= \sum_{i=0}^\nu i L_{m}(i) + O_{m}(\nu^{m}) = \frac{\nu^{m+1}}{(m+1)(m-1)!} + O_m(\nu^{m}) \\
  &= \frac{d^{(m+1)/m}}{(m+1)(m-1)!}\,\delta^{m+1} + O_{n,d}(\delta^{m}). 
\end{align*}
This yields
\[
\frac{m\sigma_i}{f} = \frac{(m+1)a_i}{d^{1/m}} + O_{n,d}(\delta^{-1}).
\]
In particular, we can choose $\delta$ depending only on $n$, $d$ and $\epsi$, such that
\[
\frac{m\sigma_i}{f} \leq \frac{(m+1)a_i}{d^{1/m}} + \epsi. 
\]
This concludes the proof of Theorem \ref{thm:auxiliary}.


\section{Integral points on affine varieties}

In this section we prove Theorem \ref{thm:affine}. Let $X \subset \AA^n_\RR$ be an irreducible closed subvariety of dimension $m$ and degree $d$, and let $I = I(X)\subset \RR[x_1,\dotsc,x_n]$. Let $\overline{X} \subset \PP^n_\RR$ be the projective closure of $X$ \cite[\S 8.4]{CLO}. Then the homogeneous ideal of $\overline{X}$ is the homogenization $I^h$ of $I$. Furthermore, let $\bB = (1,B,\dotsc,B) \in \RR^{n+1}$. If $(x_1,\dotsc,x_n) \in X(\ZZ,B)$, then $(1,x_1,\dotsc,x_n) \in S(\overline{X}, \bB)$. 

By Theorem \ref{thm:auxiliary}, given a graded monomial ordering $<$, there exists a homogeneous polynomial $G \in \ZZ[x_0,\dotsc,x_n] \setminus I^h$ of degree
\[
k \ll_{n,d,\epsi} B^{(a_1+\dotsb+a_n)(m+1)d^{-1/m}+ \epsi}
\]
such that $G(\xx) = 0$ for all $\xx \in S(\overline{X},\bB)$. 

Put
\[
g(x_1,\dotsc,x_n) = G(1,x_1,\dotsc,x_n).
\]
Then $g \notin I$ since $G \notin I^h$, and we have $g(\xx) = 0$ for every $\xx \in X(\ZZ,B)$. To complete the proof of Theorem \ref{thm:affine}, we only have to prove that $<$ can be chosen in such a way that
\begin{equation}
\label{eq:a_i}
a_1 + \dotsb + a_n \leq \frac{m}{m+1}.
\end{equation}

We shall follow the proof of \cite[Lemma 1.12]{Salberger07}, defining the graded monomial ordering by letting
$\xx^\alpha < \xx^\beta$ if and only if $|\alpha| < |\beta|$  or $|\alpha| = |\beta|$ and the left-most entry of $\alpha - \beta$ is positive. 
We need to examine, for $s \in \ZZ_{\geq 0}$, the quantity
\[
\sigma_1(s) + \dotsc + \sigma_n(s) = \sum_{\substack{\alpha \in \Lambda_{n+1}(s)\\ \xx^\alpha \notin \LT(I^h)}} (\alpha_1 + \dotsc + \alpha_n).
\]
Suppose that
\[
\{\xx^\alpha; \alpha \in \Lambda_{n+1}(s), \xx^\alpha \notin \LT(I^h) \} = \{M_1,\dotsc,M_{r}\},
\]
where $r = \HF_{I^h}(s)$. Put 
\[
m_i(x_1,\dotsc,x_n) = M_i(1,x_1,\dotsc,x_n) \text{ for } i \in \{1,\dotsc,r\}.
\]
Then we have
\begin{equation}
\label{eq:sigma=deg}
\sigma_1(s) + \dotsc + \sigma_n(s) = \deg m_1 + \dotsc +\deg m_r.
\end{equation}

Now, let $J \subset \RR[x_0,\dotsc,x_n]$ be the homogeneous ideal generated by $I^h$ and $x_0$. We claim that
\begin{equation}
\label{eq:m_i}
m_i \notin \LT(J), \text{ for } i=1,\dotsc,r.
\end{equation}
Indeed, if $m_i$ were the leading monomial of $F = x_0 G + H$, where $H \in I^h$, then $M_i = x_0^{s-\deg m_i} m_i$ would be the leading monomial of $x_0^{s-\deg m_i} H \in I^h$.

Since the $m_i$ are distinct, \eqref{eq:sigma=deg} and \eqref{eq:m_i} imply that
\[
\sigma_1(s) + \dotsc + \sigma_n(s) \leq \sum_{t=1}^s t \HF_J(t). 
\]
For the variety $X_0 = \overline X \cap H_0$ defined by $J$ we have
\[
\deg X_0 \leq \deg X, \quad \dim X_0 \leq m-1,
\] 
so
\[
\HF_J(t) \leq \frac{d}{(m-1)!} t^{m-1} + O_{n,d}(t^{m-2}),
\]
whence
\[
\sum_{t=1}^s t \HF_J(t) \leq \frac{d}{(m+1)(m-1)!} s^{m+1} + O_{n,d}(s^m).
\]
Thus we get
\[
\frac{\sigma_1(s) + \dotsc + \sigma_n(s)}{s \HF_{I^h}(s)} \leq \frac{m}{m+1} + O_{n,d}(s^{-1}).
\]
By Proposition \ref{prop:a_i} we conclude that \eqref{eq:a_i} holds. This completes the proof of Theorem \ref{thm:affine}.

\bibliographystyle{plain}
\bibliography{ratpoints}

\end{document}